\documentclass[12pt]{amsart}
\usepackage{amsmath, amsthm, amsfonts, amssymb}

\newtheorem{theorem}{Theorem}[section]
\newtheorem{lemma}[theorem]{Lemma}

\newtheorem{corollary}[theorem]{Corollary}
\theoremstyle{definition}

\theoremstyle{remark}

\numberwithin{equation}{section}

\newcommand{\mC}{\ensuremath{\mathbb{C}}}
\newcommand{\mD}{\ensuremath{\mathbb{D}}}

\newcommand{\mN}{\ensuremath{\mathbb{N}}}

 \allowdisplaybreaks
 
\begin{document}

\title{Two Normality Criteria and Counterexamples to the Converse of Bloch's Principle}

\author[K. S. Charak]{ Kuldeep Singh Charak}
\address{Department of Mathematics, University of Jammu,
Jammu-180 006, India}
\email{kscharak7@rediffmail.com}

\author[V. Singh]{Virender Singh}
\address{Department of Mathematics, University of Jammu,
Jammu-180 006, India}
\email{virendersingh2323@gmail.com }
 
\begin{abstract}
In this paper, we prove two normality
criteria for a family of meromorphic functions.
The first criterion extends a result of Fang and Zalcman[{\it Normal families and shared values of meromorphic functions II,
Comput. Methods Funct. Theory, 1(2001), 289 - 299}] to a bigger class of differential polynomials whereas the second one
leads to some counterexamples to the converse of the Bloch's principle.
\end{abstract}

\renewcommand{\thefootnote}{\fnsymbol{footnote}}
\footnotetext{The research work of the second author is supported by the CSIR India.}

\maketitle

\section{\textbf{Introduction and Main Results}}
It is assumed that the reader is familiar with the standard
notions used in the Nevanlinna value distribution theory such as
$T(r,f),m(r,f),N(r,f),S(r,f)$ etc., one may refer
to \cite{hayman-1}. In this paper, we obtain a normality criterion
for a family of meromorphic functions which involves sharing of
holomorphic functions by certain differential polynomials generated by the members of the family.\\

In 2001, Fang and Zalcman \cite[Theorem 2, p.291]{fang-1} proved the following\\

\textbf{Theorem A.} Let $\mathcal{F}$ be a family of meromorphic functions on a domain $D$, $k$ be a positive integer and $a(\neq 0)$ and $b$ be two finite values. If, for every $f\in \mathcal{F}$, all zeros of $f$ have multiplicity at least $k$ and $f(z)f^{(k)}(z)$=$a\Leftrightarrow f^{(k)}(z)$=$b$, then the family $\mathcal{F}$ is normal on $D$.\\

In this paper, we extend this result as

\begin{theorem}\label{theorem1} Let $\mathcal{F}$ be a family of meromorphic functions
on a domain $D$. Let $n\geq 2,m\geq k\geq 1$ be the positive
integers and let $a(\neq 0)$ and $b$ be two finite values.
If, for each $f\in \mathcal{F}$, $f^{n}(z)(f^{m})^{(k)}(z)$=$a\Leftrightarrow (f^{m})^{(k)}(z)$=$b$, then
the family $\mathcal{F}$ is normal on $D$.
\end{theorem}

Now it is natural to ask whether Theorem 1.1 still holds if $a$ and $b$ are holomorphic functions. In this direction, we prove the following

\begin{theorem} \label{theorem2} Let $n\geq 2,m\geq k\geq 1$ be the positive integers. Let $a(z)(\not\equiv 0)$ and $b(z)$ be two holomorphic functions on a domain $D$ such that multiplicity of each zero of $a(z)$ is at most $p$, where $p\leq \left\lceil\frac{n-1}{m}\right\rceil-1.$
Then, the family $\mathcal{F}$ of meromorphic functions on a domain $D$, all of whose poles are of multiplicity at least $p+1$, such that $f^{n}(z)(f^{m})^{(k)}(z)$=$a(z)\Leftrightarrow (f^{m})^{(k)}(z)$=$b(z)$, for every $f\in F$, is normal on $D$.
\end{theorem}
 
\textbf{Remark 1.1}. Consider the family $\mathcal{F}$=$\{$$f_l :l\in\mN\}$, where $f_l (z)=e^{lz}$ on
the unit disk $\mD$. Then
$$(f_l ^{m})^{(k)}(z)=m^k l^k e^{mlz} \text{ and } f_l ^{n}(z)(f_l ^{m})^{(k)}(z)=m^k l^k e^{(n+m)lz} $$
Clearly, $f_l ^{n}(z)(f_l ^{m})^{(k)}(z)$=$0\Leftrightarrow (f_l ^{m})^{(k)}(z)$=$0$. However, $\mathcal{F}$ is not
normal on $\mD$. Thus the condition that $a\neq0$ is essential in Theorem \ref{theorem1}.\\

\textbf{Remark 1.2}. Consider the family $\mathcal{F}$=$\{$$f_l :l\in\mN\}$, where $f_l (z)=2lz$ on the unit disk $\mD$. Then 
$$f_l ^{n}(z)(f_l ^{m})^{(k)}(z)=(2l)^{n+m} m(m-1)(m-2)....(m-k)z^{n+m-k}$$ and $$(f_l ^{m})^{(k)}(z)=(2l)^{m}m(m-1)(m-2)....(m-k)z^{m-k}$$

Clearly, $f_l ^{n}(z)(f_l ^{m})^{(k)}(z)$=$a(z)\Leftrightarrow
(f_l ^{m})^{(k)}(z)$=$b(z)$, where $a(z)=z^{n+m-k}$ and
$b(z)=z^{m-k}$. We can see that multiplicity of zeros of $a(z)$ is
at least $n$. However, the family $\mathcal{F}$ is not normal on $\mD$.
Thus, the restriction on the multiplicities of the zeros of $a(z)$ is essential in Theorem \ref{theorem2}.\\

In 2004, Lahiri and Dewan \cite[Theorem 1.4, p.3]{lahiri-2} proved\\

\textbf{Theorem B.} Let $\mathcal{F}$ be a family of meromorphic functions in a domain $D$ and $a(\neq 0),b\in \mC$.
Suppose that $E_f =\left\{ z\in D : f^{(k)} -af^{-n} =b\right\} $, where $k$ and $n(\geq k)$ are the positive integers. If for every
$f\in \mathcal{F}$ \\
$(i)$ $f$ has no zero of multiplicity less than $k$\\
$(ii)$ there exists a positive number $M$ such that for every $f\in \mathcal{F}$, $\left| f(z)\right |
 \geq M$ whenever $z\in E_f$, then $\mathcal{F}$ is normal.

\medskip

In 2006, Xu and Zhang\cite[Theorem 1.3, p.5]{xu-2} improved Theorem B as\\

\textbf{Theorem C.} Let $\mathcal{F}$ be a family of meromorphic functions in a domain $D$ and $a(\neq 0),b\in \mC$.
 Suppose that $E_f =\left\{ z\in D : f^{(k)} -af^{-n} =b\right\} $, where $k$ and $n$ are the positive integers. If for every $f\in \mathcal{F}$ \\
$(i)$ $f$ has no zero of multiplicity at least $k$\\
$(ii)$ there exists a positive number $M$ such that for every $f\in \mathcal{F}$, $\left| f(z)\right | \geq M$ whenever $z\in E_f$, then $\mathcal{F}$ is normal so long as\\
$(A)$ $n\geq 2$ or \\
$(B)$ $n=1$ and $\overline{N}_k (r,1/f)=S(r,f)$.

\medskip

In this paper, we prove the following

\begin{theorem} \label{theorem3} Let $\mathcal{F}$ be a family of meromorphic functions in a domain $D$. Let $n_1, n_2, m>k\geq1 $  be the non-negative integers such that $n_1 + n_2  \geq 1 $. Suppose $\psi (z):= f^{n_1} (z) (f^m)^{(k)}(z) -a f^{-n_2} (z) -b $, where $a(\neq 0),b \in\mC$. If there exists a positive constant $M$ such that for every $f\in \mathcal {F}$, either $\left| f(z)\right | \geq M$ or $\left| (f^m)^{(k)}(z)\right | \leq M$ whenever $z$ is a zero of $\psi (z)$, then $\mathcal{F}$ is normal in $D$.
\end{theorem}

As an application of Theorem \ref{theorem3}, we construct some counterexamples to the converse of Bloch's principle in the last section of this paper.

\begin{corollary} Let $\mathcal{F}$ be a family of meromorphic functions in a domain $D$. Let $n,m>k$  be the positive integers and $a(\neq0)$ be a finite complex number. If there exists a positive constant $M$ such that for every $f\in \mathcal {F},~f^{n} (z) (f^m)^{(k)}(z) =a \Rightarrow \left| (f^m)^{(k)}(z)\right | \leq M$, then $\mathcal{F}$ is normal in $D$.
\end{corollary}

\section{\textbf{Some Lemmas}}

\begin{lemma}\cite{zalcman-1}\label{lemma1} (Zalcman's lemma) Let $\mathcal{F}$ be a family of meromorphic functions in the unit disk $\mD$ and $\alpha$ be a real number satisfying $-1<\alpha<1$. Then, if $\mathcal{F}$ is not normal at a point $z_0 \in \mD $, there exist, for each $\alpha:-1<\alpha<1$,\\
$(i)$ a real number $r$: $r <1$,\\
$(ii)$ points $z_n$: $\left|z_n\right| <r$,\\
$(iii)$ positive numbers $\rho_n$: $\rho_n$$\rightarrow$0,\\
$(iv)$ functions  $f_n\in F$
such that $g_n$($\zeta$)=$\rho^{-\alpha}$$f_n$($z_n$+$\rho_n$$\zeta$) converges locally uniformly with respect to the spherical metric to $g$($\zeta$), where $g(\zeta$) is a non constant meromorphic function on $\mC$ and $g^{\#}$($\zeta$)$\leq g^{\#}(0)=1$.
Moreover, the order of $g$ is not greater than 2.
\end{lemma}

\begin{lemma}\cite[Lemma 2.6, p.107]{zeng-1} \label{lemma2}Let $R=\frac{A}{B}$ be a rational function and $B$ be non constant. Then $(R^{(k)})_{\infty}\leq (R)_{\infty}-k$,  where $(R)_{\infty}$=$deg(A)-deg(B)$.
\end{lemma}

\begin{lemma}\label{lemma3} Let $n\geq2$, $m\geq k\geq 1$ be the positive integers. Let $a(z)(\not\equiv 0)$ be a polynomial of degree $p$ such that $p\leq n-2$. Then there is no function $f$ rational on $\mathbb{C}$ which has only poles of multiplicity at least $p+1$ such that
$f^{n}(z)(f^{m})^{(k)}(z) \neq a(z)$ and $ (f^{m})^{(k)}(z) \neq 0$.
\end{lemma}

\begin{proof} First we consider the case of a polynomial. Suppose on the contrary that there is a polynomial $f(z)$ with the given properties. Since $(f^{m})^{(k)}\neq 0$ and $m \geq k$, $f$ has zeros of multiplicity exactly one. So, we have
$$\text{deg}(f^n (f^m )^{(k)}) \geq n\text{deg}(f)=n>p=\text{deg}(a(z))$$
Therefore, $f^{n}(z)(f^{m})^{k}(z)- a(z)$ has a solution, which is a contradiction.\\
Next, suppose that $f$ has poles. Then, we set
\begin{equation}f(z)= A \frac{\prod\limits_{i=1}^{s}{(z-{{\alpha }_{i}})}}{\prod\limits_{j=1}^{t}{{{(z-{{\beta }_{j}})}^{{{n}_{j}}}}}},
\end{equation}
where $A \neq 0, ~\alpha_i $ are the distinct zeros of $f$ with $s\geq 0$ and $\beta_j $ are the distinct poles of $f$ with $t\geq 1$.\\
Put
$$\sum\limits_{j=1}^{t}{ n_j}=N.$$
Then $$N\geq t(p+1).$$
Now,
\begin{align}
f^m (z)&=A^m \frac{\prod\limits_{i=1}^{s}{(z-{{\alpha }_{i}})^{m}}}{\prod\limits_{j=1}^{t}{{{(z-{{\beta }_{j}})}^{{{mn}_{j}}}}}}\\
\Rightarrow(f^m)^{(k)}(z)&= \frac{\prod\limits_{i=1}^{s}{(z-{{\alpha }_{i}})^{m-k}}}{\prod\limits_{j=1}^{t}{{{(z-{{\beta }_{j}})}^{{{mn}_{j} +k}}}}}g(z),
\end{align}
where $g(z)$ is a polynomial.\\
By Lemma \ref{lemma2}, we have
\begin{align*}
&(f^m)_{\infty} ^{(k)} \leq (f^m)_\infty -k \\
\Rightarrow &deg (g)\leq k(s+t-1).
\end{align*}
Now,
\begin{align}\label{1}
f^n (f^m)^{(k)}&=A^n\frac{\prod\limits_{i=1}^{s}{(z-{{\alpha }_{i}})^{(m+n)-k}}}{\prod\limits_{j=1}^{t}{{{(z-{{\beta }_{j}})}^{{{(m+n)n}_{j}+k}}}}}g(z).
\end{align}
So,
\begin{align}\label{2}
(f^n (f^m)^{(k)})^{(p+1)}&=\frac{\prod\limits_{i=1}^{s}{(z-{{\alpha }_{i}})^{(m+n)-k-p-1}}}{\prod\limits_{j=1}^{t}{{{(z-{{\beta }_{j}})}^{{{(m+n)n}_{j} +k+p+1}}}}}g_0(z),
\end{align}
where  $g_0 (z)$ is a polynomial.\\
Again, by Lemma \ref{lemma2}, we have
\begin{align*}
(f^n (f^m)^{(k)})^{(p+1)} _\infty  &\leq  (f^n (f^m)^{(k)})_\infty -(p+1) \\
\Rightarrow deg(g_0) &\leq(s+t-1)(p+k+1). 
\end{align*}
Since $ f^n (f^m)^{(k)}\neq a(z)$, we set
\begin{align}\label{3}
f^n (f^m)^{(k)}&= a(z) + \frac{c}{\prod\limits_{j=1}^{t}{{{(z-{{\beta }_{j}})}^{{{(m+n)n}_{j} +k}}}}}, \end{align}
 where $ c\neq0 $ is a constant.\\
So,
\begin{align} \label{4}
 (f^n (f^m)^{(k)})^{(p+1)}& =\frac{g_1(z)}{\prod\limits_{j=1}^{t}{{{(z-{{\beta }_{j}})}^{{{(m+n)n}_{j} +k+p+1}}}}},
\end{align}
where $g_1 (z)$ is a polynomial of degree at most $ (p+1)(t-1)$.\\
On comparing (\ref{1}) and (\ref{3}), we have
\begin{align*}
s(m+n)-ks+deg(g)&=N(m+n)+kt+pt \\
\Rightarrow N(m+n) &\leq s(m+n)-k  \\
\Rightarrow N &<s,\end{align*} for $n\geq 2,m\geq k\geq 1$.\\
Also, from (\ref{2}) and (\ref{4}),  we have
         $$deg(g_1)\geq s(m+n)-s(k+p+1).$$
Now,
\begin{align*}
(p+1)(t-1)&\geq \text{deg} (g_1 (z))\geq s(m+n)-s(k+p+1) \\
\Rightarrow s(m+n)&\leq (p+1)(t-1)+s(k+p+1)  \\
\Rightarrow s(m+n)& < (p+1)t+s(k+p+1)  \\
\Rightarrow s&< \frac{p+1}{m+n} t+ \frac{k+p+1}{m+n}s  \\
\Rightarrow s&< \frac{1}{m+n} N + \frac{k+p+1}{m+n}s  \\
\Rightarrow s&< \left(\frac{1}{m+n} + \frac{k+p+1}{m+n} \right)s  \\
\Rightarrow s&< \left( \frac{k+p+2}{m+n}\right)s \\
\Rightarrow s&<s~~~~~\left(\because \frac{k+p+2}{m+n}\leq 1 \right),
\end{align*}
which is absurd.\\
Thus, if $(f^m)^{(k)}(z)\neq 0$, then $f^n (z) (f^m )^{(k)}(z)-a(z)$ has at least a solution. Hence the Lemma follows.
\end{proof}

\begin{lemma}\label{lemma4} Let $ n\geq 2,m\geq k\geq 1$ be the positive integers. Then there is no transcendental meromorphic function $f$ on $\mathbb{C}$ such that $f^{n}(z)(f^{m})^{(k)}(z) \neq a(z)$ and $(f^{m})^{(k)}(z) \neq 0$, where $a(z)\not\equiv 0$ is a small function of $f$.\end{lemma}

\begin{proof}  Suppose on the contrary that there is a transcendental meromorphic function $f$ on $\mathbb{C}$ satisfying the given conditions. Since $(f^m)^{(k)} \neq0$ and $m\geq k$, $f$ has zeros of multiplicity exactly one. Now, by second fundamental theorem of Nevanlinna for three small functions\cite[Theorem 2.5, p.47]{hayman-1}, we have
\begin{align}
T(r,f^n (f^m)^{(k)} )&\leq \overline{N} (r,f^n (f^m)^{(k)})+\overline{N} \left(r,\frac{1}{f^n (f^m)^{(k)}}\right)+\overline{N} \left(r,\frac{1}{f^n (f^m)^{(k)}-a(z)}\right) \notag \\
&=\overline{N}(r,f)+\overline{N}\left(r,\frac{1}{f}\right)+S(r,f).\label{5}
\end{align}
Also,
\begin{align}
T(r,f^n (f^m)^{(k)} )&\geq \frac{1}{2}\left[N(r,f^n (f^m)^{(k)})+N\left(r,\frac{1}{f^n (f^m)^{(k)}}\right)  \right] \notag \\
&\geq   \frac{n+m+k }{2}\overline{N}(r,f)+\overline{N}\left(r,\frac{1}{f}\right)+S(r,f). \label{6}
\end{align}
Thus, from (\ref{5}) and (\ref{6}), we get
\begin{align}
\frac{n+m+k }{2}\overline{N}(r,f)   &\leq   \overline{N}(r,f)+  S(r,f) \notag \\
\Rightarrow \overline{N}(r,f)&= S(r,f).\label{7}
\end{align}
Next,\begin{align}
(m+n)T(r,f)&=T(r,f^{m+n}) \notag  \\
&=T\left(r,\frac{1}{f^{m+n}}\right) +O(1) \notag  \\
&=m\left(r,\frac{1}{f^{m+n}}\right)+N\left(r,\frac{1}{f^{m+n}}\right)+O(1) \notag \\
&=m\left(r,\frac{(f^m)^{(k)}}{f^{m}} \frac{1}{f^n (f^m)^{(k)}}\right)+N\left(r,\frac{1}{f^{m+n}}\right)+O(1) \notag \\
&\leq m\left(r,\frac{1}{f^n (f^m)^{(k)}}\right)+N\left(r,\frac{1}{f^{m+n}}\right)+O(1) \notag \\
&\leq T(r,f^n (f^m)^{(k)})-N\left(r,\frac{1}{f^n (f^m)^{(k)}}\right)+N\left(r,\frac{1}{f^{m+n}}\right)+S(r,f).\label{8}
\end{align}
Now, substituting (\ref{5}) and (\ref{7}) in (\ref{8}), we get													
\begin{align*}
(m+n)T(r,f)&\leq \overline{N}\left(r,\frac{1}{f}\right)-N \left(r,\frac{1}{f^n (f^m)^{(k)}}\right)+N\left(r,\frac{1}{f^{m+n}}\right)+S(r,f)\\
&\leq \overline{N}\left(r,\frac{1}{f}\right)-n\overline{N}\left(r,\frac{1}{f} \right)+(m+n)\overline{N}\left(r,\frac{1}{f}\right)+S(r,f)\\
&=(m+1)\overline{N}\left(r,\frac{1}{f}\right)+S(r,f)\\
&\leq(m+1)T(r,f)+S(r,f)\\
\Rightarrow(n-1)T(r,f)&\leq S(r,f),
\end{align*}
which is a contradiction, for $n\geq 2$.\\
However, if $f$ has no zeros, then $f^n(f^m)^{(k)}$ has no zeros. \\That is,
$$ N\left(r,\frac{1}{f}\right)=S(r,f) ~\text{and}~N\left(r,\frac{1}{f^n (f^m)^{(k)}}\right)=S(r,f).$$
Thus, by the same argument used above, we get a contradiction.
\end{proof}

\begin{lemma} \cite{charak-2}\label{lemma5} Let $f$ be a transcendental meromorphic function and $n,m>k$ be the positive integers. Let $F=f^n (f^m )^{(k)}$. Then $$\left[ \frac{k}{2(2k+2)}+o(1) \right ] T(r,F)\leq \overline{N} \left(r,\frac{1}{F-\omega}\right) +S(r,F)$$ for any small function $\omega(\not\equiv 0,\infty$) of $f$.
\end{lemma}

\begin{lemma} \cite{charak-2}\label{lemma6} Let $f$ be a rational function and $n,m>k$ be the positive integers. Then, for $a(\neq 0) \in \mathbb{C}$, $f^n (f^m )^{(k)}-a$ has at least two distinct zeros.
\end{lemma}

\begin{lemma}\cite{clunie-1}\label{lemma7} Let $f$ be an entire function. If the spherical derivative $f^{\#}$ is bounded in $\mathbb{C}$, then the order of $f$ is at most one.
\end{lemma}

\section{\textbf{Proof of Theorems}}
		
\begin{proof} [\textbf{Proof of Theorem 1.1.}]
Suppose that $\mathcal{F}$ is not normal at some point $z_o\in D$. We assume $D=\mD$. Then by Lemma \ref{lemma1}, we can find a sequence $\left\{ f_j \right\}$ in $\mathcal{F}$, a sequence $\left\{ z_j\right\}$ of complex numbers with $z_j\rightarrow z_o$ and a sequence $\left\{\rho_j\right\}$ of positive real numbers with $\rho_j \rightarrow 0$ such that
$$g_j (\zeta) =\rho_j ^{\frac{-k}{n+m}} f_j (z_j +\rho_j \zeta)$$
converges locally uniformly with respect to the spherical metric to a non-constant meromorphic function $g(\zeta)$ on $\mC$ having bounded spherical derivative.

\medskip

{\bf Claim}:
\begin{enumerate}
	\item $g^n (g^{m})^{(k)}\neq a $\\
	\item  $(g^m )^{(k)} \neq 0$
\end{enumerate}

\medskip
Suppose that $g^n (\zeta_o ) (g^{m})^{(k)} (\zeta_o )= a $. Then $g(\zeta_o)\neq\infty$  in some small neighborhood of $\zeta_o$.
Further, $g^n (g^m)^{(k)} \not\equiv a $. Suppose $g^n (g^m)^{(k)} \equiv a $. Since $g$ is a non-constant entire function without zeros, by Lemma \ref{lemma7}, we have $g(\zeta)=e^{c\zeta+d}$, where $c\neq0$ and $d$ are constants. Thus $$m^k c^k e^{(m+n)c\zeta+(m+n)d} \equiv a$$
which is impossible unless $(m+n)c=0$. Hence by Hurwitz theorem, there exist points $\zeta_j \rightarrow \zeta_o $ such that, for sufficiently large $j$, we have
$$a=g^n _j (\zeta_j ) (g^m _j )^{(k)} (\zeta_j) =f^n _j (\zeta_j + \rho_j \zeta_j ) (f^m _j )^{(k)} (\zeta_j + \rho_j \zeta_j).$$
By given condition, we have  
$$(f^m _j)^{(k)} (\zeta_j + \rho_j \zeta_j)=b, $$ 
and hence,
$$(g^m _j)^{(k)} (\zeta_j )= \rho_j ^{\frac{nk}{m+n}} (f^m _j)^{(k)} (z_j + \rho_j \zeta_j)=\rho_j ^{\frac{nk}{m+n}} b $$
$$\Rightarrow (g^m )^{(k)} (\zeta_o)= \underset{j\to \infty }{\mathop{\lim }}\, (g^m _j)^{(k)} (\zeta_j )  =0$$ \\
which contradicts that $g^n (\zeta_o ) (g^{m})^{(k)} (\zeta_o )= a\neq 0$. This proves claim (1).

\medskip
Now, suppose $(g^m )^{(k)} (\zeta_o)=0$ for some $\zeta_o\in\mathbb{C}$, then $g(\zeta_o) \neq \infty $ in some small neighborhood of $\zeta_o$. Further, $(g^m )^{(k)}\not \equiv 0$, otherwise, $g$ reduces to a constant since $m\geq k$.
Again, by Hurwitz theorem, there exist points $\zeta_j \rightarrow \zeta_o $ such that, for sufficiently large $j$, we have
\begin{align*} &(g^m _j)^{(k)} (\zeta_j )-\rho_j ^{\frac{nk}{m+n}} b =0  \\
&\Rightarrow\rho_j ^{\frac{nk}{m+n}} (f^m _j)^{(k)} (z_j + \rho_j \zeta_j)-\rho_j ^{\frac{nk}{m+n}} b=0  \\
&\Rightarrow(f^m _j)^{(k)} (z_j + \rho_j \zeta_j)=b. 
\end {align*}
Thus, by the given condition, we get 
$$f^n _j (z_j + \rho_j \zeta_j ) (f^m _j )^{(k)} (z_j + \rho_j \zeta_j)= a=g^n _j (\zeta_j ) (g^m _j )^{(k)} (\zeta_j)$$
$$\Rightarrow a =\underset{j\to \infty }{\mathop{\lim }}\,g^n _j (\zeta_j ) (g^m _j )^{(k)} (\zeta_j)=g^n (\zeta_o ) (g^m  )^{(k)} (\zeta_o)=0$$
which is a contradiction. This proves claim (2). \\
Claims (1) and (2) as established contradict Lemma \ref{lemma3} and Lemma \ref{lemma4}. Hence $\mathcal{F} $ is normal.
\end{proof}

\begin{proof} [\textbf{Proof of Theorem 1.2.}]
Suppose that $\mathcal{F}$ is not normal at some point $z_o \in D$. We assume $D=\mD$. We distinguish the following two cases:

   \medskip
\textbf{Case I}: $a(z_o)\neq 0$\\
Following the proof of Theorem 1.1, we arrive at a contradiction and hence $\mathcal{F}$ is normal in this case.

    \medskip
\textbf{Case II}: $a(z_o)= 0$\\
Without loss of generality, we assume that $z_o =0$. Further, we assume $a(z)=z^p a_1 (z) $, where $p$ is a positive integer and $a_1 (0)\neq 0$. We may take $a_1 (0)=1$.
Now, by Lemma \ref{lemma1}, we can find a sequence $\left\{ f_j \right\}$ in $\mathcal{F}$, a sequence $\left\{ z_j\right\}$ of complex numbers with $z_j\rightarrow 0$ and a sequence $\left\{\rho_j\right\}$ of positive real numbers with $\rho_j \rightarrow 0 $ such that $$g_j (\zeta) =\rho_j ^{- \frac{p+k}{n+m}} f_j (z_j +\rho_j \zeta)$$ converges locally uniformly with respect to the spherical metric to a non-constant meromorphic function $g(\zeta)$ on $\mC$ having bounded spherical derivative.\\\\
\textbf{Subcase I}: Suppose there exist a subsequence of $\frac{z_j}{\rho_j}$, we may take  $\frac{z_j}{\rho_j}$ itself, such that  $\frac{z_j}{\rho_j} \rightarrow \infty $ as $j \rightarrow \infty$.\\
Let $$G_j (\zeta) =\ z_j ^{- \frac{p+k}{n+m}} f_j (z_j +z_j \zeta).$$
Then, by the given condition $f^n (z)(f^{m})^{(k)}(z)=a(z) \Leftrightarrow  (f^{m})^{(k)}(z) =b(z)$, we have
$$G_j ^n (\zeta)(G_j ^{m})^{(k)} (\zeta)=(1+\zeta)^p a_1 (z_j + z_j \zeta) \Leftrightarrow (G_j ^{m})^{(k)} (\zeta)=z_j ^l b (z_j +z_j \zeta ),$$
where $$l=-\frac{m(p+k)}{n+m} +k>0.$$
Thus, by Case I, $\left\{G_j\right\}$ is normal on $\mD$ and $G_j\rightarrow G $ (say) on $\mD$. Hence, by Marty's theorem, there exist a compact subset E of $\mD$ and a constant M$>0$ such that
$$G_j ^\# (\xi)\leq M ~\text{for}~ \xi\in E.$$
Claim: $G^ \# (0)=0$. Suppose $G^ \# (0)\neq 0$. Then for $\zeta\in \mathbb{C}$, we have
\begin{align*}g^\# (\zeta) &= \underset{j\to \infty }{\mathop{\lim }}\ g_j ^\# (\zeta)  \\
&=\underset{j\to \infty }{\mathop{\lim }}\ \rho_j ^{- \frac{p+k}{n+m}} f_j ^\# (z_j +\rho_j \zeta)  \\
&=\underset{j\to \infty }{\mathop{\lim }}\   \left(\frac{z_j}{\rho_j}\right)^{ \frac{p+k}{n+m}} G_j ^\# \left(\frac{\rho_j}{z_j} \zeta \right)\\
&=\infty 
\end{align*}
which is a contradiction to the fact that $g$ has bounded spherical derivative.\\
Now, $G^ \# (0)=0 \Rightarrow G'(0)=0$.
For any $\zeta \in \mC$, we have
\begin{align*} g' _j (\zeta) & = \rho_j ^{-\frac{p+k}{n+m}+1} f' _j (z_j+\rho_j \zeta)\\
& = \left( \frac{\rho_j}{z_j}\right)^{-\frac{p+k}{n+m}+1} G'_j \left( \frac{\rho_j}{z_j} \zeta \right) \overset{\chi }{\mathop{\to }}\, 0 
\end {align*}
on $\mathbb{C}$ as $\frac{p+k}{n+m}<1$. Thus $g'(\zeta)\equiv 0$ implies that $g$ is constant and this is a contradiction.\\\\
\textbf{Subcase II}: Suppose there exist a subsequence of $\frac{z_j}{\rho_j}$, we may take  $\frac{z_j}{\rho_j}$ itself, such that  $\frac{z_j}{\rho_j}\rightarrow c $ as $j \rightarrow \infty$, where $c$ is a finite number.\\
Then, we have
$$H_j (\zeta) =\rho_j ^{- \frac{p+k}{n+m}} f_j (\rho_j \zeta)=g_j \left( \zeta -\frac{z_j}{\rho_j}\right)\overset{\chi }{\mathop{\to }}\, g(\zeta -c):=H(\zeta).$$
Thus, by the given condition, we have
$$H_j ^n (\zeta)(H_j ^{m})^{(k)} (\zeta)=\zeta^p a_1 (\rho_j \zeta) \Leftrightarrow (H_j ^{m})^{(k)} (\zeta)=\rho_j ^l b (\rho_j \zeta ),$$
where $$l=-\frac{m(p+k)}{n+m} +k>0.$$

\medskip
{\bf Claim}:
\begin{enumerate}
 \item $H^n (\zeta) (H^{m})^{(k)} (\zeta) \neq \zeta^p $ on $\mathbb{C}-\left\{ 0 \right\}$\\
 \item $(H^m )^{(k)} (\zeta) \neq 0$ on $\mathbb{C}-\left\{ 0 \right\}$
\end{enumerate}
\medskip
Suppose that $H^n (\zeta_o) (H^{m})^{(k)} (\zeta_o)= \zeta_o ^p $, $\zeta_o \neq 0$. Then, $H(\zeta_o)\neq\infty$ on some small neighborhood of $\zeta_o $. Further, $H^n (\zeta) (H^{m})^{(k)} (\zeta) \not\equiv \zeta^p $. If $H^n (\zeta) (H^{m})^{(k)} (\zeta)\equiv \zeta^p$, then $\zeta=0$ is the only possible zero of $H$.
If $H$ is a transcendental function, then, clearly $H^n (H^m )^{(k)}$ is also a transcendental function, which is not true.
If $H$ is a rational function and $\zeta=0$ is a zero of $H$, then $H$ is a polynomial. Thus, deg$(H^n (H^m )^{(k)}) \geq n$deg$(H) \geq n$, which is a contradiction to the fact that $H^n (\zeta) (H^{m})^{(k)} (\zeta)\equiv \zeta^p$, $p\leq n-2$.
By Hurwitz's theorem, there exist points $\zeta_j \rightarrow \zeta_o $ such that, for sufficiently large $j$, we have
\begin{align*}& H_j ^n (\zeta_j)(H_j ^{m})^{(k)} (\zeta_j)-\zeta_j ^p a_1 (\rho_j \zeta_j)=0\\
\Rightarrow &(H_j ^{m})^{(k)} (\zeta_j)-\rho_j ^l  b (\rho_j \zeta_j )=0.
\end{align*}
Thus,\begin{align*} (H^m)^{(k)}( \zeta_o)&= \underset{j\to \infty }{\mathop{\lim }}\ (H_j ^m)^{(k)} (\zeta_j) \\
&= \underset{j\to \infty }{\mathop{\lim }}\ \rho_j ^l  b (\rho_j \zeta_j )\\
&=0 
\end{align*}
which contradicts that $H^n (\zeta_o) (H^{m})^{(k)} (\zeta_o)= \zeta_o ^p \neq 0$ . This proves claim (1).
 
\medskip
Next, suppose $(H^m )^{(k)} (\zeta_o)=0$ for some $\zeta_o\in\mC-\left\{0 \right\}$. Then $H(\zeta_o) \neq \infty $ on some small neighborhood of $\zeta_o$. Further, $(H^m )^{(k)}\not \equiv 0$, otherwise, $H$ reduces to a constant since $m\geq k$.
Thus, by Hurwitz theorem, there exist points $\zeta_j \rightarrow \zeta_o $ such that, for sufficiently large $j$, we have
\begin{align*} &(H^m _j)^{(k)} (\zeta_j )-\rho_j ^l   b (\rho_j \zeta_j) =0  \\
\Rightarrow &H^n _j (\zeta_j ) (H^m _j )^{(k)} (\zeta_j) -\zeta_j ^p a_1(\rho_j \zeta_j )=0 
\end {align*}
and so
\begin{align*}
H^n  (\zeta_o) (H^m )^{(k)} (\zeta_o)&=\underset{j\to \infty }{\mathop{\lim }}\ H^n _j (\zeta_j ) (H^m _j )^{(k)} (\zeta_j) \\
&=\underset{j\to \infty }{\mathop{\lim }}\ \zeta_j ^p a_1(\rho_j \zeta_j )  \\
&=\zeta_o ^p 
\end{align*}  which is a contradiction. This proves claim (2).\\
Claims (1) and (2) as established contradict Lemma \ref{lemma3} and Lemma \ref{lemma4}. Hence $\mathcal{F}$ is normal.
\end{proof}

\begin{proof}[\textbf{Proof of Theorem 1.3.}]
Suppose that $\mathcal{F}$ is not normal at some point $z_0 \in D$. Then by Lemma \ref{lemma1}, we can find a sequence $\left\{ f_j \right\}$ in $\mathcal{F}$, a sequence $\left\{ z_j\right\}$ of complex numbers with $z_j\rightarrow z_o$ and a sequence $\left\{\rho_j\right\}$ of positive real numbers with $\rho_j \rightarrow 0$ such that $$g_j (\zeta) =\rho_j ^{\frac{-k}{n_1 +n_2 +m}} f_j (z_j +\rho_j \zeta)$$ converges locally uniformly with respect to the spherical metric to a non-constant meromorphic function $g(\zeta)$ on $\mC$ having bounded spherical derivative.
Now, by Lemma \ref{lemma5} and Lemma \ref{lemma6}, $g^n (\zeta) (g^m)^{(k)} (\zeta)-a$ has at least one zero for $n\geq1,m>k\geq 1$. Suppose that $g^n (\zeta_0)  (g^m)^{(k)} (\zeta_0)-a =0$ for some $\zeta_0 \in \mathbb{C}$. Clearly, $g(\zeta_0)\neq 0,\infty$ in some neighborhood of $\zeta_0$. Thus, we have
$$g^{n_1} (\zeta_0)  (g^m)^{(k)} (\zeta_0)-a g^{-n_2} (\zeta_0) =0,$$ where $ n=n_1 + n_2 \geq1$.\\
Now, in some neighborhood of $\zeta_0$, we have\\
$ g_j ^{n_1} (\zeta_0)  (g_j ^m)^{(k)} (\zeta_0)-a g_j ^{-n_2} (\zeta_0) - \rho_j ^{\frac{kn_2}{n+m}} b$ $$=\rho_j ^{\frac{kn_2}{n+m}} \left\{f_j ^{n_1} (\zeta_j + \rho_j \zeta_0)  (f_j ^m)^{(k)} (\zeta_j + \rho_j \zeta_0)
-a f_j ^{-n_2} (\zeta_j + \rho_j \zeta_0) - b\right\} $$
By Hurwitz's theorem, there exists a sequence $\zeta_j \rightarrow \zeta_0$ such that for all large values of $j$,
$$f_j ^{n_1} (\zeta_j + \rho_j \zeta_j)  (f_j ^m)^{(k)} (\zeta_j + \rho_j \zeta_j)-a f_j ^{-n_2} (\zeta_j + \rho_j \zeta_j)-b=0$$
Thus, by the assumption, if $\left| f_j (\zeta_j + \rho_j \zeta_j ) \right|\geq M$, then we have
$$\left | g_j (\zeta_j) \right| =  \rho_j ^{\frac{-k}{n+m}}\left| f_j (\zeta_j + \rho_j \zeta_j ) \right|  \geq \rho_j ^{\frac{-k}{n+m}}M.$$
Since $g_j (\zeta)$ converges uniformly to $g(\zeta)$ in some neighborhood of $\zeta_0$, for all large values of $j$ and for every $\epsilon >0$, we have
$$\left | g_j (\zeta) - g(\zeta) \right|<\epsilon ~~\text{~for all}~ \zeta ~\text{in that neighborhood of}~ \zeta_o.$$
Thus, in a neighborhood of $\zeta_o$, for all large values of $j$, we have
\begin{align*} \left|g(\zeta_j) \right| &\geq \left |g_j (\zeta_j ) \right| - \left | g(\zeta_j) - g_j (\zeta_j ) \right| \\
& > \rho_j ^{\frac{-k}{n+m}}M - \epsilon 
\end {align*}
which is a contradiction to the fact that $\zeta_0$ is not a pole of $g(\zeta)$.\\
Again, by the assumption, if $\left| (f_j ^m)^{(k)}(z_j + \rho_j \zeta_j)\right | \leq M$, then we have  
 $$\left| (g_j ^m)^{(k)}( \zeta_j)\right | =\rho_j ^{k-\frac{mk}{n_1 + n_2 + m}} \left| (f_j ^m)^{(k)}(z_j + \rho_j \zeta_j)\right | \leq\rho_j ^{k-\frac{mk}{n_1 + n_2 + m}} M $$
so that $$(g^m)^ {(k)} (\zeta_o) = \underset{j\to \infty }{\mathop{\lim }} (g_j ^m )^{(k)}(\zeta_j)=0 $$
which contradicts $g^n (\zeta_o) (g^m)^ {(k)} (\zeta_o) =a\neq 0$. Hence $\mathcal{F}$ is normal.

\end{proof}

\section{\textbf{Counterexamples to the converse of the Bloch's Principle}}
The Bloch's principle  as noted by Robinson \cite{robinson-1} is one of the twelve mathematical problems requiring further consideration; it is a heuristic principle in function theory. The Bloch's principle states that a family of holomorphic (meromorphic) functions satisfying a property $\mathcal P$ in a domain $D$ is likely to be a normal family if the property $\mathcal P$ reduces every holomorphic (meromorphic) function on $\mC$ to a constant. The Bloch's principle is not universally true, for example one can see \cite{rubel-1}.

\medskip 

The converse of the Bloch's Principle states that if a family of meromorphic functions satisfying a property $\mathcal P$ on an arbitrary domain $D$ is necessarily a normal family, then every meromorphic function on $\mathbb{C}$ with property $\mathcal P$ reduces to a constant. Like Bloch's principle, its converse is not true. For counterexamples one can see \cite{charak-1},\cite{lahiri-1},\cite{li-1},\cite{schiff-1},\cite{xu-1},\cite{yunbo-1}. In order to construct counterexamples to the converse, one needs to prove a suitable normality criterion. Here Theorem \ref{theorem3} is such a criterion. Infact, following is a direct consequence of Theorem \ref{theorem3}:
\begin{theorem}\label{theorem4}  Let $\mathcal{F}$ be a family of meromorphic functions in a domain $D$. Let $n_1, n_2, m>k\geq1 $  be the non-negative integers such that $n_1 + n_2  \geq 1 $. Suppose $\psi (z):= f^{n_1} (z) (f^m (z))^{(k)} -a f^{-n_2} (z) -b $, where $a(\neq 0),b \in\mathbb{C}$, has no zeros in $D$. Then $\mathcal{F}$ is normal in $D$.
\end{theorem}  
Now by Theorem \ref{theorem4}, we have the following four counterexamples to the converse of the Bloch's principle: 

\medskip

Consider $f(z)=e^z$. Then for $n_1 =1,~ n_2 = 0,~ m=2,~ k=1, ~a=-1, \mbox { and } ~b=1$, $\psi (z):=f(z)(f^2)' (z)+1-1=2e^{3z}$ has no zeros in $\mC$. Thus there is a non constant entire function with property $\mathcal{ P}:\psi(z)$ has no zeros in $\mC$. Hence in view of Theorem \ref{theorem4}, this is a counterexample to the converse of Bloch's principle.

\medskip

Similarly, for the same values of the constants $n_1 , n_2 , m, k, a, \mbox { and } b$, the meromorphic functions 
$$\frac{1}{z},\ \ \ \ \frac{1}{e^z +1},\ \ \ \text{tan}z \pm i ,$$  provide three more counterexamples to the converse of the Bloch's principle.

\bibliographystyle{amsplain}

\end {document}